\documentclass[12pt,a4paper,reqno]{amsart} %optio titlepage!

\usepackage[T1]{fontenc}       % kirjaimet, joissa aksentteja (skandit)

\usepackage{amsfonts}          % AMS-fontit

\usepackage{amsmath}

\usepackage{amssymb}

\usepackage{overpic}
\usepackage{euscript} % \mathcal tuottaa hyvдnnдkцisen fontin
\usepackage{eurosym}
\usepackage{comment}
\usepackage{mathrsfs} % calligraphic font via the command \mathscr{ABC}
\usepackage{ifthen}
\usepackage{esint} %k.a.integraali komennolla \fint
\usepackage{color}

\long\def\symbolfootnote[#1]#2{\begingroup%
\def\thefootnote{\fnsymbol{footnote}}\footnote[#1]{#2}\endgroup}

{\qed\vspace{5pt}}

\usepackage{amsthm}

\newtheoremstyle{lause}% name
{5pt}% space above
{5pt}% space below
{\slshape}% body font
{\parindent}% indent amount (empty = no indent)
{\bfseries}% theorem head font
{.}% punctuation after theorem head
{.5em}% space after theorem head
{}% theorem head spec (can be left empty, meaning 'normal')

\theoremstyle{lause}
\newtheoremstyle{maaritelma}% name
{5pt}% space above
{5pt}% space below
{\rmfamily}% body font
{\parindent}% indent amount (empty = no indent)
{\bfseries}% theorem head font
{.}% punctuation after theorem head
{.5em}% space after theorem head
{}% theorem head spec (can be left empty, meaning 'normal')

\theoremstyle{maaritelma}
\newtheoremstyle{lause}% name
{5pt}% space above
{5pt}% space below
{\slshape}% body font
{\parindent}% indent amount (empty = no indent)
{\bfseries}% theorem head font
{.}% punctuation after theorem head
{.5em}% space after theorem head
{}% theorem head spec (can be left empty, meaning 'normal')

\theoremstyle{lause}
\newtheorem{theorem}{Theorem}[section]
\newtheorem{lemma}[theorem]{Lemma}
\newtheorem{proposition}[theorem]{Proposition}

\newtheoremstyle{maaritelma}% name
{5pt}% space above
{5pt}% space below
{\rmfamily}% body font
{\parindent}% indent amount (empty = no indent)
{\bfseries}% theorem head font
{.}% punctuation after theorem head
{.5em}% space after theorem head
{}% theorem head spec (can be left empty, meaning 'normal')

\theoremstyle{maaritelma}
\newtheorem{definition}[theorem]{Definition}

\newtheorem{remark}[theorem]{Remark}

\DeclareMathOperator{\cp}{cap}

\numberwithin{equation}{section}

\begin{document}

\thispagestyle{empty}

\begin{center}

{\large{\textbf{Balayage of Radon measures of infinite energy\\ on locally compact spaces}}}

\vspace{18pt}

\textbf{Natalia Zorii}

\vspace{18pt}

\emph{Dedicated to Professor Wolfgang Dahmen on the occasion of his 75th birthday}\vspace{8pt}

\footnotesize{\address{Institute of Mathematics, Academy of Sciences
of Ukraine, Tereshchenkivska~3, 01601,
Kyiv-4, Ukraine\\
natalia.zorii@gmail.com }}

\end{center}

\vspace{12pt}

{\footnotesize{\textbf{Abstract.} For suitable kernels on a locally compact space $X$, we develop a theory of inner balayage of quite general Radon measures $\omega$ (not necessarily of finite energy) to arbitrary $A\subset X$. In the case where $A$ is Borel, this theory provides, as a by-product, a theory of outer balayage. We prove the existence and the uniqueness of inner (outer) swept measures, analyze their properties, and provide a number of alternative characterizations. In spite of being in agreement with Cartan's theory of Newtonian balayage, the results obtained require essentially new methods and approaches, since in the case in question, useful specific features of Newtonian potentials may fail to hold. The theory thereby established generalizes substantially the existing ones, pertaining either to $\omega$ of finite energy, or to some particular $A$ (e.g.\ quasiclosed). This work covers many interesting kernels in classical and modern potential theory, which looks promising for possible applications.}}
\symbolfootnote[0]{\quad 2010 Mathematics Subject Classification: Primary 31C15.}
\symbolfootnote[0]{\quad Key words: Radon measures on a locally compact space; function kernels; potentials; energy, consistency, and maximum principles; inner and outer balayage.
}

\vspace{6pt}

\markboth{\emph{Natalia Zorii}} {\emph{Balayage of Radon measures of infinite energy on locally compact spaces}}

\section{Introduction and general conventions}\label{sec-intr}

The present paper deals with the theory of potentials on a locally compact (Hausdorff) space $X$ with respect to a kernel $\kappa$, {\it a kernel\/} being thought of as a symmetric, lower semi\-con\-tin\-uous (l.s.c.) function $\kappa:X\times X\to[0,\infty]$.

In more details, under suitable requirements on a kernel $\kappa$, we develop a theory of {\it inner} balayage $\omega^A$ of quite a general Radon measure $\omega$ (not necessarily of finite energy) to {\it arbitrary} $A\subset X$ (Theorems~\ref{th1}, \ref{th2}, \ref{pr-in-ou}, \ref{l-alt-countt}). In the case where $A$ is Borel, the theory thereby established provides, as a by-pro\-d\-uct, a theory of {\it outer} balayage $\omega^{*A}$ (Theorems~\ref{th1'}, \ref{pr-in-ou}, \ref{l-alt-countt}). Regarding the theory of outer balayage, cf.\ also Theorem~4.12 by Fuglede \cite{Fu5}, pertaining to {\it quasiclosed} sets, that is, to those $A\subset X$ which can be approximated in outer capacity by closed sets \cite{F71}. (It is worth emphasizing here that a quasiclosed set is not necessarily Borel, and vice versa.)

As to the history of the question, the theory of inner balayage of any $\omega$ of {\it finite} energy to any $A\subset X$ was developed in the author's recent papers \cite{Z-arx1}--\cite{Z-arx}, and this was performed for any $\kappa$ satisfying the energy, consistency, and domination principles.

Such a theory was further extended in \cite{Z-24} to quite general measures $\omega$ whose energy might already be {\it infinite}. To this end, we needed, however, to impose upon $X$, $\kappa$, and $A$ some additional assumptions. In particular, all positive measures of finite energy, concentrated on the set $A$, were required in \cite{Z-24} to form a strongly closed cone. (This would occur, for instance, if $A$ were quasiclosed, see \cite[Theorem~2.13]{Z-Oh}.)

In this work, we generalize the theory established in \cite{Z-24} to {\it arbitrary} $A\subset X$. Moreover, some of the results thereby obtained are new even in the case treated in \cite{Z-24}. In particular, we show that if the space $X$ has a countable base of open sets, then the inner balayage to {\it arbitrary} $A$ can always be reduced to the balayage to some {\it Borel} $B\subset A$, and hence $\omega^A=\omega^B=\omega^{*B}$ (see Theorem~\ref{pr-in-ou}; compare with Proposition~VI.2.2 by Bliedtner and Hansen \cite{BH}, related to outer balayage). See also Theorem~\ref{l-alt-countt}, specifying the theory of inner (outer) balayage to the case where any $\varphi\in C_0(X)$, a continuous function on $X$ of compact support, can be approximated in the inductive limit topology on $C_0(X)$ by potentials of measures of finite energy.

\subsection{General conventions}\label{sec-intr1} To begin with, we shall first introduce some notions and notations, to be used throughout this paper.

Given a locally compact space $X$, we denote by $\mathfrak M$ the linear space of all (real-valued Radon) measures $\mu$ on $X$ equipped with the {\it vague} topology of pointwise convergence on the class $C_0(X)$ of all continuous functions $\varphi:X\to\mathbb R$ of compact support ${\rm Supp}(\varphi)$, and by $\mathfrak M^+$ the cone of all positive $\mu\in\mathfrak M$, where $\mu$ is {\it positive} if and only if $\mu(\varphi)\geqslant0$ for all positive $\varphi\in C_0(X)$. (For the theory of measures and integration on a locally compact space, we refer to Bourbaki \cite{B2} or Edwards \cite{E2}.)

Given $\mu,\nu\in\mathfrak M$, the {\it mutual energy} and the {\it potential} are introduced by
\begin{align*}
  I(\mu,\nu)&:=\int\kappa(x,y)\,d(\mu\otimes\nu)(x,y),\\
  U^\mu(x)&:=\int\kappa(x,y)\,d\mu(y),\quad x\in X,
\end{align*}
respectively, provided the value on the right is well defined as a finite number or $\pm\infty$. For $\mu=\nu$, the mutual energy $I(\mu,\nu)$ defines the {\it energy} $I(\mu,\mu)=:I(\mu)$ of $\mu\in\mathfrak M$.

In what follows, a kernel $\kappa$ is assumed to satisfy the {\it energy principle}, or equivalently to be {\it strictly positive definite}, which means that $I(\mu)\geqslant0$ for all (signed) $\mu\in\mathfrak M$, and moreover that $I(\mu)$ equals $0$ only for the zero measure. Then all $\mu\in\mathfrak M$ of finite energy $0\leqslant I(\mu)<\infty$ form a pre-Hil\-bert space $\mathcal E$ with the inner product $\langle\mu,\nu\rangle:=I(\mu,\nu)$ and the energy norm $\|\mu\|:=\sqrt{I(\mu)}$, cf.\ \cite[Lemma~3.1.2]{F1}. The topology on $\mathcal E$ determined by means of this norm, is said to be {\it strong}.

In addition, we always assume that $\kappa$ satisfies the {\it consistency} principle, which means that the cone
$\mathcal E^+:=\mathcal E\cap\mathfrak M^+$ is {\it complete} in the induced strong topology, and that the strong topology on $\mathcal E^+$ is {\it finer} than the induced vague topology on $\mathcal E^+$; such a kernel is said to be {\it perfect} (Fuglede \cite{F1}). Thus any strong Cauchy sequence (net) $(\mu_j)\subset\mathcal E^+$ converges {\it both strongly and vaguely} to the same unique measure $\mu_0\in\mathcal E^+$, the strong topology on $\mathcal E$ as well as the vague topology on $\mathfrak M$ being Hausdorff.

Yet another permanent requirement upon $\kappa$ is that it satisfies  the {\it domination} and {\it Ugaheri maximum principles}, where the former means that for any $\mu\in\mathcal E^+$ and any $\nu\in\mathfrak M^+$ with $U^\mu\leqslant U^\nu$ $\mu$-a.e., the same inequality holds on all of $X$; whereas the latter means that there is $h\in[1,\infty)$, depending on $X$ and $\kappa$ only, such that for each $\mu\in\mathcal E^+$ with $U^\mu\leqslant c_\mu$ $\mu$-a.e., where $c_\mu\in(0,\infty)$, we have $U^\mu\leqslant hc_\mu$ on all of $X$. When $h$ is specified, we speak of {\it $h$-Ugaheri's maximum principle}, and when $h=1$, $h$-Ugaheri's maximum principle is referred to as {\it Frostman's maximum principle} \cite{O}.

For any $A\subset X$, we denote by $\mathfrak C_A$ the upward directed set of all compact subsets $K$ of $A$, where $K_1\leqslant K_2$ if and only if $K_1\subset K_2$. If a net $(x_K)_{K\in\mathfrak C_A}\subset Y$ converges to $x_0\in Y$, $Y$ being a topological space, then we shall indicate this fact by writing
\begin{equation*}x_K\to x_0\text{ \ in $Y$ as $K\uparrow A$}.\end{equation*}

Given $A\subset X$, we denote by $\mathfrak M^+(A)$ the cone of all $\mu\in\mathfrak M^+$ {\it concentrated on}
$A$, which means that $A^c:=X\setminus A$ is locally $\mu$-neg\-lig\-ible, or equivalently that $A$ is $\mu$-meas\-ur\-able and $\mu=\mu|_A$, $\mu|_A:=1_A\cdot\mu$ being the {\it trace} of $\mu$ to $A$ \cite[Section~IV.14.7]{E2}. (Note that for $\mu\in\mathfrak M^+(A)$, the indicator function $1_A$ of $A$ is locally $\mu$-int\-egr\-able.) The total mass of $\mu\in\mathfrak M^+(A)$ is $\mu(X)=\mu_*(A)$, $\mu_*(A)$ and $\mu^*(A)$ denoting the {\it inner} and the {\it outer} measure of $A$, respectively. If moreover $A$ is closed, or if $A^c$ is contained in a countable union of sets $Q_j$ with $\mu^*(Q_j)<\infty$,\footnote{If the latter holds, $A^c$ is said to be $\mu$-$\sigma$-{\it finite} \cite[Section~IV.7.3]{E2}. This in particular occurs if the measure $\mu$ is {\it bounded} (that is, with $\mu(X)<\infty$), or if the locally compact space $X$ is {\it $\sigma$-com\-pact} (that is, representable as a countable union of compact sets \cite[Section~I.9, Definition~5]{B1}).} then for any $\mu\in\mathfrak M^+(A)$, $A^c$ is $\mu$-neg\-lig\-ible, that is, $\mu^*(A^c)=0$. In particular, if $A$ is closed, $\mathfrak M^+(A)$ consists of all $\mu\in\mathfrak M^+$ with support ${\rm Supp}(\mu)\subset A$, cf.\ \cite[Section~III.2.2]{B2}.

We also define $\mathcal E^+(A):=\mathcal E\cap\mathfrak M^+(A)$. As seen from \cite[Lemma~2.3.1]{F1},\footnote{For the {\it inner} and {\it outer} capacities of $A\subset X$, denoted by $\cp_*A$ and $\cp^*A$, respectively, we refer to \cite[Section~2.3]{F1}. If $A$ is capacitable (e.g.\ open or compact), we write $\cp A:=\cp_*A=\cp_*A$.}
\begin{equation}\label{iff}
 \cp_*A=0\iff\mathcal E^+(A)=\{0\}\iff\mathcal E^+(K)=\{0\}\quad\text{for all $K\in\mathfrak C_A$}.
\end{equation}

In what follows, fix {\it arbitrary} $A\subset X$. To avoid trivialities, suppose that
\begin{equation*}
 \cp_*A>0.
\end{equation*}
When approximating $A$ by $K\uparrow A$, we may therefore only consider $K$ with $\cp K>0$.

Also fix $\omega\in\mathfrak M^+$, $\omega\ne0$. Referring to \cite{Z-arx1}--\cite{Z-arx} for the theory of balayage of any $\omega\in\mathcal E^+$ to any $A\subset X$, established for any perfect kernel satisfying the domination principle, in the present study we shall allow $I(\omega)$ to be $+\infty$.

Then, along with the above-mentioned permanent requirements on $\kappa$, i.e.
\begin{itemize}
\item[(a)] {\it $\kappa$ is perfect, and satisfies the domination and $h$-Ugaheri maximum principles,}
\end{itemize}
we shall also assume that (b) and (c) are fulfilled, where:
\begin{itemize}
  \item[(b)] $\omega$ {\it is bounded.}
  \item[(c)] $U^\omega$ {\it is continuous on every compact subset of $A$, and it is bounded on $A$, i.e.}
  \begin{equation}\label{supA}
   \sup_{x\in A}\,U^\omega(x)<\infty.
  \end{equation}
\end{itemize}
The permanent assumptions (a)--(c) will usually not be repeated henceforth.

\begin{remark}\label{(a)}
Assumption (a) is fulfilled, for instance, for the following kernels:
\begin{itemize}
  \item[$\checkmark$] The $\alpha$-Riesz kernels $|x-y|^{\alpha-n}$ of order $\alpha\in(0,2]$, $\alpha<n$, on $\mathbb R^n$, $n\geqslant2$ (see  \cite[Theorems~1.10, 1.15, 1.18, 1.27, 1.29]{L}).
  \item[$\checkmark$] The associated $\alpha$-Green kernels, where $\alpha\in(0,2]$ and $\alpha<n$, on an arbitrary open subset of $\mathbb R^n$, $n\geqslant2$ (see \cite[Theorems~4.6, 4.9, 4.11]{FZ}).
   \item[$\checkmark$] The ($2$-)Green kernel, associated with the Laplacian, on a planar Greenian set (see \cite[Theorem~5.1.11]{AG}, \cite[Sections~I.V.10, I.XIII.7]{Doob}, and \cite{E}).
\end{itemize}
For all those kernels, $h=1$, that is, Frostman's maximum principle actually holds.
\end{remark}

\section{Inner and outer balayage}\label{sec-main}

\subsection{Inner balayage}\label{sec-main-inner}
Given an arbitrary set $A\subset X$, let $\mathcal E'(A)$ stand for the closure of $\mathcal E^+(A)$ in the strong topology on $\mathcal E^+$.\footnote{If $A=:F$ is closed, then, due to the perfectness of the kernel $\kappa$ and the vague closedness of the class $\mathfrak M^+(F)$ \cite[Section~III.2, Proposition~6]{B2}, we have $\mathcal E'(F)=\mathcal E^+(F)$. More generally, this equality remains valid if $F$ is quasiclosed, see \cite[Theorem~2.13]{Z-Oh}.} Being a strongly closed subcone of the strongly complete cone $\mathcal E^+$ (Section~\ref{sec-intr1}), the cone $\mathcal E'(A)$ is likewise strongly complete.

Then for any given $\lambda\in\mathcal E^+$, there exists the only measure $\lambda^A\in\mathcal E'(A)$ such that\footnote{An assertion $\mathcal A(x)$ involving a variable point $x\in X$ is said to hold {\it nearly everywhere} ({\it n.e.}) on a set $Q\subset X$ if the set $N$ of all $x\in Q$ where $\mathcal A(x)$ fails, is of inner capacity zero. Replacing here $\cp_*N=0$ by $\cp^*N=0$, we obtain the concept of {\it quasi-everywhere} ({\it q.e.}) on $Q$. See \cite[p.~153]{F1}.}
\begin{equation}\label{n.e.}
U^{\lambda^A}=U^\lambda\quad\text{n.e.\ on $A$};
\end{equation}
this $\lambda^A$ is called the {\it inner balayage} of $\lambda$ to $A$, see \cite[Theorem~4.3]{Z-arx1}.
See also \cite{Z-arx1}--\cite{Z-arx} for a number of alternative characterizations of the inner balayage $\lambda^A$, some of which will be quoted in the course of proofs given below.

At the moment, we only note that $\lambda^A$ coincides with the {\it orthogonal projection} of $\lambda\in\mathcal E^+$ in the pre-Hil\-bert space $\mathcal E$ onto the (convex, strongly complete) cone $\mathcal E'(A)$, see \cite[Theorem~4.3]{Z-arx1}. (With regard to the orthogonal projection in a pre-Hil\-bert space, see \cite[Theorem~1.12.3, Proposition~1.12.4]{E2}.) 
In particular, this implies that
\begin{equation}\label{L}
\lambda^A=\lambda\quad\text{for any $\lambda\in\mathcal E'(A)$}.
\end{equation}

Returning to $\omega\in\mathfrak M^+$ whose energy might be $+\infty$, we shall now introduce the concept of inner balayage as follows.

\begin{definition}\label{def-in} The {\it inner balayage} of $\omega\in\mathfrak M^+$ to any $A\subset X$ is understood as $\omega^A\in\mathcal E'(A)$ meeting the symmetry relation
\begin{equation}\label{sym}I(\omega^A,\lambda)=I(\lambda^A,\omega)\quad\text{for all $\lambda\in\mathcal E^+$},\end{equation}
where $\lambda^A$ denotes the only measure in $\mathcal E'(A)$ satisfying {\rm(\ref{n.e.})}.
\end{definition}

\begin{remark}\label{rem-agree}
The concept of inner balayage thus defined agrees with that by Cartan \cite[Section~18]{Ca2} or that by the author \cite[Definition~3.9]{Z-bal}, pertaining to the Newtonian or Riesz kernels, respectively.
\end{remark}

\begin{lemma}\label{uniq-in}
The inner balayage $\omega^A$ is unique {\rm(}if it exists{\rm)}.
\end{lemma}

\begin{proof}
Assume (\ref{sym}) is also fulfilled for some $\nu\in\mathcal E'(A)$ in place of $\omega^A$, i.e.
\begin{equation}\label{sym'}I(\nu,\lambda)=I(\lambda^A,\omega)\quad\text{for all $\lambda\in\mathcal E^+$}.\end{equation}
Note that the left-hand side in either of (\ref{sym}) or (\ref{sym'}) is finite, hence so is that on the right.
Subtracting (\ref{sym'}) from (\ref{sym}) therefore gives $I(\omega^A-\nu,\lambda)=0$ for all $\lambda\in\mathcal E^+$,
and consequently for all $\lambda\in\mathcal E$. Taken for $\lambda:=\omega^A-\nu$, this yields $I(\omega^A-\nu)=0$, whence $\omega^A=\nu$, the kernel $\kappa$ being strictly positive definite.
\end{proof}

\begin{remark}It will be shown in Theorem~\ref{th1} below that under the permanent requirements (a)--(c), $\omega^A$ does exist, and it can alternatively be defined by any one of a number of equivalent characteristic properties (cf.\ also Theorem~\ref{th2}). If moreover the space $X$ is second-countable, while the set of all $\varphi\in C_0(X)$ representable as potentials of signed measures of finite energy is dense in $C_0(X)$ equipped with the inductive limit topology,\footnote{Regarding the inductive limit topology on the space $C_0(X)$, see Bourbaki \cite[Section~II.4.4]{B4} and \cite[Section~III.1.1]{B2} (cf.\ also Section~\ref{sec-prel} below).} then the symmetry relation (\ref{sym}) needs only to be verified for certain countably many $\lambda\in\mathcal E^+$, depending on $X$ and $\kappa$ only (Theorem~\ref{l-alt-countt}).\end{remark}

\subsection{On the existence and alternative characterizations of $\omega^A$}\label{sec-main1} To formulate the main results of this paper, we start with some notations. Setting
\begin{equation}\label{H}
H:=H_{\kappa,\omega}:=h\omega(X)\in(0,\infty),
\end{equation}
$h$ appearing in $h$-Ugaheri's maximum principle (see (a) and (b)), we denote\footnote{The class $\mathcal E^+_H(A)$ would certainly be the same if $\mu(X)$ in (\ref{eH}) were replaced by $\mu_*(A)$, since for any $\mu\in\mathfrak M^+(A)$, $\mu(X)=\mu|_A(X)=\mu_*(A)$.}
\begin{equation}\label{eH}
\mathcal E^+_H(A):=\bigl\{\mu\in\mathcal E^+(A): \ \mu(X)\leqslant H\bigr\},
\end{equation}
and let $\mathcal E'_H(A)$ stand for the strong closure of $\mathcal E^+_H(A)$.
Being a strongly closed subset of the strongly complete cone $\mathcal E'(A)$, $\mathcal E'_H(A)$ is likewise {\it strongly complete}.

Furthermore, in view of the perfectness of $\kappa$, each $\zeta\in\mathcal E'_H(A)$ is both the strong and the vague limit of some net (sequence) $(\mu_j)\subset\mathcal E^+_H(A)$, whence
\begin{equation}\label{supm}
\mathcal E'_H(A)\subset\mathcal E^+_H(\overline{A}),\quad\text{where $\overline{A}:={\rm Cl}_XA$},
\end{equation}
the class $\mathfrak M^+(\overline{A})$ being vaguely closed \cite[Section~III.2, Proposition~6]{B2}, whereas the mapping $\nu\mapsto\nu(X)$ being vaguely l.s.c.\ on $\mathfrak M^+$ \cite[Section~IV.1, Proposition~4]{B2}.

It is often convenient to treat the above $\omega$ as a charge creating the {\it external field}
\[f:=-U^\omega.\]
As seen from (\ref{supA}), this $f$ is bounded on $\overline{A}$; therefore, the so-called {\it Gauss functional\/}\footnote{For the terminology used here, see \cite{L,O}. In constructive function theory, $I_f(\mu)$ is sometimes
referred to as {\it the $f$-weighted energy}, see e.g.\ \cite{BHS,Dr0,ST}.}
\begin{equation}\label{Gf}
I_f(\mu):=\|\mu\|^2+2\int f\,d\mu=\|\mu\|^2-2\int U^\omega\,d\mu
\end{equation}
is {\it finite} for all bounded $\mu\in\mathcal E^+(\overline{A})$, and hence for all $\mu\in\mathcal E'_H(A)$ (cf.\ (\ref{H})--(\ref{supm})).

\begin{theorem}\label{th1}
Under the permanent assumptions {\rm(a)--(c)}, the inner balayage $\omega^A$, introduced by means of Definition~{\rm\ref{def-in}}, does exist, is unique, and it can alternatively be characterized by any one of the following equivalent assertions {\rm(i)--(iv)}.
\begin{itemize}
\item[{\rm(i)}] $\omega^A$ is the only measure in $\mathcal E'_H(A)$ having the property
\begin{equation}\label{n.e.'}
U^{\omega^A}=U^\omega\quad\text{n.e.\ on $A$}.
\end{equation}
\item[{\rm(ii)}] $\omega^A$ is uniquely determined within $\mathcal E'_H(A)$ by any one of the limit relations\footnote{Assertion (ii) justifies the term "{\it inner} balayage".}
\begin{align}
  \omega^K&\to\omega^A\quad\text{strongly in $\mathcal E^+$ as $K\uparrow A$},\label{c1}\\
  \omega^K&\to\omega^A\quad\text{vaguely in $\mathfrak M^+$ as $K\uparrow A$},\label{c2}\\
  U^{\omega^K}&\uparrow U^{\omega^A}\quad\text{pointwise on $X$ as $K\uparrow A$},\label{c3}
\end{align}
where $\omega^K$ denotes the only measure in $\mathcal E^+(K)$ having the property
\begin{equation}\label{balK}U^{\omega^K}=U^\omega\quad\text{n.e.\ on $K$},\end{equation}
or equivalently
\begin{equation}\label{GaussK}I_f(\omega^K)=\min_{\mu\in\mathcal E^+(K)}\,I_f(\mu)=\min_{\mu\in\mathcal E^+_H(K)}\,I_f(\mu),\end{equation}
the Gauss functional $I_f(\cdot)$ being introduced by means of {\rm(\ref{Gf})}.
\item[{\rm(iii)}] $\omega^A$ is the only measure in the class $\Gamma_{A,\omega}$ having the property\footnote{Relations (\ref{minpot}) and (\ref{minen}) would obviously be the same if the class $\Gamma_{A,\omega}$ were replaced by either of $\Gamma_{A,\omega}\cap\mathcal E'_H(A)$ or $\Gamma_{A,\omega}\cap\mathcal E^+_H(\overline{A})$.}
\begin{equation}\label{minpot}
U^{\omega^A}=\min_{\nu\in\Gamma_{A,\omega}}\,U^\nu\quad\text{on $X$},
\end{equation}
where
\begin{equation}\label{G}
\Gamma_{A,\omega}:=\bigl\{\nu\in\mathcal E^+:\ U^\nu\geqslant U^\omega\quad\text{n.e.\ on $A$}\bigr\}.
\end{equation}
\item[{\rm(iv)}] $\omega^A$ is the only measure in the class $\Gamma_{A,\omega}$ having the property
\begin{equation}\label{minen}
\|\omega^A\|=\min_{\nu\in\Gamma_{A,\omega}}\,\|\nu\|,
\end{equation}
$\Gamma_{A,\omega}$ being introduced by means of {\rm(\ref{G})}.
\end{itemize}
\end{theorem}

\begin{remark}\label{REM}
 It will be clear from the proof of Theorem~\ref{th1} (Section~\ref{sec-proof1}) that, similarly as it occurred for $\lambda\in\mathcal E^+$, (\ref{n.e.'}) characterizes $\omega^A$ uniquely within $\mathcal E'(A)$.
\end{remark}

Keeping the permanent requirement (a), in the following Theorem~\ref{th2} we also assume that (d)--(f) are fulfilled, where:
\begin{itemize}
\item[(d)] {\it $\kappa(x,y)$ is continuous for $x\ne y$.}
\item[(e)] {\it $\kappa(\cdot,y)\to0$ uniformly on compact subsets of $X$ when $y\to\infty_X$.} (Here, $\infty_X$ denotes the Alexandroff point of $X$, see \cite[Section~I.9.8]{B1}.)
\item[(f)] {\it ${\rm Supp}(\omega)$ is compact, and moreover ${\rm Supp}(\omega)\cap\overline{A}=\varnothing$.}
\end{itemize}
Note that then, both (b) and (c) do hold automatically, and hence can be omitted.

 \begin{theorem}\label{th2} Under the requirements {\rm(a)} and {\rm(d)--(f)}, the inner balayage $\omega^A$, uniquely determined by means of either of Definition~{\rm\ref{def-in}} or Theorem~{\rm\ref{th1}}, can alternatively be characterized as follows:
\begin{itemize}
\item[{\rm(v)}] $\omega^A$ is the unique solution to the problem of minimizing $I_f(\mu)$ over $\mu\in\mathcal E'_H(A)$. That is, $\omega^A\in\mathcal E'_H(A)$ and
    \begin{equation}\label{hatw}
     I_f(\omega^A)=\min_{\mu\in\mathcal E'_H(A)}\,I_f(\mu).
    \end{equation}
\end{itemize}
\end{theorem}

\begin{remark}If $X$ is second-countable while $\mathcal E^+(A)$ is strongly closed (which occurs e.g.\ if $A$ is quasiclosed \cite[Theorem~2.13]{Z-Oh}), then, by virtue of \cite{Z-24} (see Theorem~1.2(iv) and Eq.~(2.22) therein), Theorem~\ref{th2} remains valid under the (weaker) assumptions (b) and (c) in place of (d)--(f). Moreover, then $\omega^A\in\mathcal E^+_H(A)$ and
\[I_f(\omega^A)=\min_{\mu\in\mathcal E^+_H(A)}\,I_f(\mu)=
\min_{\mu\in\mathcal E^+_f(A)}\,I_f(\mu),\]
where $\mathcal E^+_f(A)$ consists of all $\nu\in\mathcal E^+(A)$ such that $f$ is $\nu$-integrable. For applications of these results to minimum energy problems with external fields, see \cite[Section~4]{Z-24}.
\end{remark}

\subsection{Outer balayage}\label{sec-main2} In this subsection, a locally compact space $X$ is assumed to be $\sigma$-compact and {\it perfectly normal}.\footnote{By Urysohn's theorem \cite[Section~IX.4, Theorem~1]{B3}, a Hausdorff topological space $Y$ is said to be {\it normal} if for any two disjoint closed sets $F_1,F_2\subset Y$, there exist disjoint open sets $D_1,D_2\subset Y$ such that $F_i\subset D_i$ $(i=1,2)$. Further, a normal space $Y$ is said to be {\it perfectly normal} if each closed subset of $Y$ is a countable intersection of open sets, see \cite[Exercise~7 to Section~IX.4]{B3}.} It is worth noting that a sufficient condition for this to occur is that $X$ be second-countable, but not the other way around.\footnote{Indeed, a locally compact space is second-countable if and only if it is metrizable and $\sigma$-com\-pact \cite[Section~IX.2, Corollary to Proposition~16]{B3}. Being therefore metrizable, a sec\-ond-count\-able locally compact space $X$ is perfectly normal \cite[Chapter~IX]{B3} (see Section~2, Proposition~7 and Section~4, Proposition~2), whereas the converse is false even in the case of a compact space, see \cite[Exercise~13(b) to Section~IX.2]{B3}.\label{FOOT}}

Assume, in addition, that $A\subset X$ is {\it Borel}. Then for every $\lambda\in\mathcal E^+$, there exists $\lambda^{*A}$, the {\it outer balayage} of $\lambda$ to $A$, uniquely determined within $\mathcal E'(A)$ by the property
\begin{equation}\label{eq-bal-ouu}
U^{\lambda^{*A}}=U^\lambda\quad\text{q.e.\ on $A$}
\end{equation}
(see \cite[Theorem~9.4]{Z-arx1}), and moreover, according to the same theorem,
\begin{equation}\label{BB}
 \lambda^{*A}=\lambda^A.
\end{equation}

The concept of outer balayage, introduced below for $\omega\in\mathfrak M^+$, agrees with that of outer Newtonian balayage by Cartan \cite[Section~18]{Ca2}.

\begin{definition}\label{def-outer} The {\it outer balayage} of $\omega\in\mathfrak M^+$ to $A$ is defined as $\omega^{*A}\in\mathcal E'(A)$ meeting the symmetry relation
\begin{equation}\label{sym''}
I(\omega^{*A},\lambda)=I(\lambda^{*A},\omega)\quad\text{for all $\lambda\in\mathcal E^+$},
\end{equation}
where $\lambda^{*A}$ denotes the only measure in $\mathcal E'(A)$ satisfying (\ref{eq-bal-ouu}).
\end{definition}

Observe that this definition differs from that of inner balayage (see Definition~\ref{def-in}) only by replacing an exceptional set in (\ref{n.e.}) by that of {\it outer} capacity zero, cf.\ (\ref{eq-bal-ouu}).

\begin{lemma}\label{uniq-ou}
The outer balayage $\omega^{*A}$ is unique {\rm(}if it exists{\rm)}.
\end{lemma}

\begin{proof}
  This follows in exactly the same manner as Lemma~\ref{uniq-in}.
\end{proof}

\begin{theorem}\label{th1'} If a locally compact space $X$ is $\sigma$-compact and perfectly normal, then for any Borel set $A\subset X$, Theorems~{\rm\ref{th1}} and {\rm\ref{th2}} remain valid with $\omega^A$ and "n.e.\ on $A$" replaced throughout by $\omega^{*A}$ and "q.e.\ on $A$", respectively. Actually,
\begin{equation}\label{eq-bal-oUU}\omega^{*A}=\omega^A.\end{equation}
\end{theorem}

\begin{proof} As a direct application of \cite[Theorem~4.5]{F1}, we find the following conclusion.

\begin{theorem}\label{l-top}Any Borel subset of a $\sigma$-compact, perfectly normal, locally compact space $X$, endowed with a perfect kernel $\kappa$, is capacitable.
\end{theorem}

Since $A$ is Borel, so is each of the exceptional sets $N$ appearing in Theorems~\ref{th1} and \ref{th2}. Therefore, by virtue of Theorem~\ref{l-top}, $\cp^*N=\cp_*N=0$, and hence $\omega^A$ fulfils each of the assertions (i)--(v) with "q.e.\ on $A$" in place of "n.e.\ on $A$". To complete the proof, it is thus enough to show that $\omega^A$ also serves as the outer balayage $\omega^{*A}$. Indeed, substituting (\ref{BB}) into (\ref{sym}) yields (\ref{sym''}) with $\omega^{*A}$ replaced by $\omega^A$, which on account of Lemmas~\ref{uniq-in} and \ref{uniq-ou} implies (\ref{eq-bal-oUU}), whence the theorem.
\end{proof}

\section{Proofs of Theorems~\ref{th1} and \ref{th2}}\label{sec-proof}

We quote for future reference known or easily verified facts, useful in the sequel.

\begin{lemma}\label{l0}For any $\mu,\nu\in\mathcal E$ with $U^\mu=U^\nu$ q.e.\ on $X$, we have $\mu=\nu$.\end{lemma}

\begin{proof}On account of \cite[Corollary to Lemma~3.2.3]{F1} and the countable subadditivity of outer capacity \cite[Lemma~2.3.5]{F1}, we observe that $\mu-\nu$ is a (signed) measure of finite energy whose potential is well defined and equals $0$ q.e.\ on $X$. Therefore, by virtue of \cite[Lemma~3.2.1(a)]{F1}, $\|\mu-\nu\|=0$, whence $\mu=\nu$ by the energy principle.\end{proof}

\begin{lemma}\label{l1}
For any $\mu\in\mathcal E^+(E)$, where a set $E\subset X$ is $\mu$-$\sigma$-compact, and any universally measurable $U\subset X$ such that $\cp_*(E\cap U)=0$, we have $\mu^*(E\cap U)=0$.
\end{lemma}

\begin{proof}
Being the intersection of universally measurable $U$ and $\mu$-measurable $E$, the set $E\cap U$ is $\mu$-measurable. Besides, $E\cap U$ is $\mu$-$\sigma$-compact, since so is $E$. It is therefore enough to show that $\mu_*(E\cap U)=0$, which is however obvious from (\ref{iff}).
\end{proof}

\begin{lemma}\label{str}
For any $E\subset X$ and any universally measurable $U_j\subset X$, $j\in\mathbb N$,
\[\cp_*\Bigl(\bigcup_{j\in\mathbb N}\,E\cap U_j\Bigr)\leqslant\sum_{j\in\mathbb N}\,\cp_*(E\cap U_j).\]\end{lemma}

\begin{proof}Since a strictly positive definite kernel is pseudo-positive, cf.\ \cite[p.~150]{F1}, the lemma follows directly from Fuglede \cite{F1} (see Lemma~2.3.5 and the remark after it). For the Newtonian kernel on $\mathbb R^n$, this goes back to Cartan \cite[p.~253]{Ca2}.\end{proof}

\begin{lemma}\label{l2}If a net $(\nu_s)\subset\mathcal E$ converges strongly to $\nu_0$, then there exists a subsequence $(\nu_k)$ whose potentials converge to $U^{\nu_0}$ pointwise n.e.\ on $X$.\end{lemma}

\begin{proof}Since the strong topology on $\mathcal E$ is first-countable, there exists a subsequence $(\nu_j)$ of the net $(\nu_s)$ that also converges strongly to $\nu_0$. Therefore, applying to such a sequence $(\nu_j)$ the remark attached to Lemma~3.2.4 in \cite{F1}, we arrive at the claim.\end{proof}

\begin{theorem}[Principle of positivity of mass]\label{pr-pos} Assume $X$ is $\sigma$-compact and Frostman's maximum principle holds.
For any $\mu,\nu\in\mathcal E^+$, then
\begin{equation*} U^\mu\leqslant U^\nu\quad\text{n.e.\ on $X$}\Longrightarrow\mu(X)\leqslant\nu(X).\end{equation*}
\end{theorem}

\begin{proof}See Zorii \cite[Theorem~2.1]{Z-arx-22}, cf.\ Deny \cite{D2}.\footnote{Compare with \cite[Theorem~1.2]{Z-Deny}, providing quite a surprising version of the principle of positivity of mass for the $\alpha$-Riesz kernels $|x-y|^{\alpha-n}$ of order $0<\alpha\leqslant2$, $\alpha<n$, on $\mathbb R^n$, $n\geqslant2$.}\end{proof}

\subsection{Proof of Theorem~\ref{th1}}\label{sec-proof1}
It is clear from \cite{Z-24} (see (i), (ii), (iv) in Theorem~1.2, as well as Eq.~(2.14) and the sentence following it) that for every $K\in\mathfrak C_A$, the inner balayage $\omega^K$, introduced by Definition~\ref{def-in} above, does exist, and it is uniquely characterized within $\mathcal E^+(K)$ by (\ref{balK}) or, equivalently, (\ref{GaussK}).\footnote{In \cite{Z-24}, the space $X$ is required to be sec\-ond-cou\-n\-t\-able, which is, however, superfluous when dealing with compact sets.}
Thus, for these $\omega^K$,
\begin{equation}\label{BalK}\int U^{\omega^K}\,d\lambda=\int U^{\lambda^K}\,d\omega\quad\text{for all $\lambda\in\mathcal E^+$},\end{equation}
where $\lambda^K$ is the only measure in $\mathcal E^+(K)$ meeting (\ref{n.e.}) with $A:=K$.

Furthermore, these $\omega^K$, $K\in\mathfrak C_A$, form a strong Cauchy net in $\mathcal E^+_H(A)$, see the text around Eq.~(2.15) in \cite{Z-24}. As $\mathcal E^+_H(A)\subset\mathcal E'_H(A)$ while $\mathcal E'_H(A)$ is strongly complete (see Section~\ref{sec-main1} above), there exists the unique measure $\zeta\in\mathcal E'_H(A)$ such that
\begin{equation}\label{oconv}
 \omega^K\to\zeta\quad\text{strongly (hence vaguely) in $\mathcal E'_H(A)$ as $K\uparrow A$},
\end{equation}
the strong topology on $\mathcal E$ as well as the vague topology on $\mathfrak M$ being Hausdorff.

We claim that the same $\zeta$ is uniquely determined within $\mathcal E'_H(A)$ by the relation
\begin{equation}\label{potconv}
 U^{\omega^K}\uparrow U^\zeta\quad\text{pointwise on $X$ as $K\uparrow A$}.
\end{equation}
In fact, it is seen from (\ref{balK}) that for any $K'\geqslant K$, we have $U^{\omega^K}=U^{\omega^{K'}}$ n.e.\ on $K$, hence $\omega^K$-a.e.\ (Lemma~\ref{l1}), and so $U^{\omega^K}\leqslant U^{\omega^{K'}}$ on $X$ (the domination principle). Therefore, the net $(U^{\omega^K})$ increases pointwise on all of $X$.

On the other hand, $\omega^K\to\zeta$ strongly, which implies, by use of Lemma~\ref{l2}, that
\[U^{\omega^K}\uparrow U^\zeta\quad\text{pointwise n.e.\ on $X$ as $K\uparrow A$}.\]
Thus, $U^\zeta\geqslant U^{\omega^K}$ n.e.\ on $K$, hence $\omega^K$-a.e., and so, by the domination principle,
\begin{equation}\label{g''}
U^\zeta\geqslant\lim_{K\uparrow A}\,U^{\omega^K}\quad\text{on all of $X$}.
\end{equation}
Since the opposite is obvious from the vague lower semicontinuity of the mapping $\nu\mapsto U^\nu(\cdot)$ on $\mathfrak M^+$ (the principle of descent \cite[Lemma~2.2.1(b)]{F1}), equality actually prevails in (\ref{g''}), whence (\ref{potconv}).

Assume now that (\ref{potconv}) also holds for some $\theta\in\mathcal E'_H(A)$ in place of $\zeta$; then obviously $U^\theta=U^\zeta$ on $X$, and applying Lemma~\ref{l0} therefore gives $\theta=\zeta$ as claimed.

To complete the proof of (ii), it remains to show that the above $\zeta$ serves as the inner balayage $\omega^A$, or equivalently (cf.\ Definition~\ref{def-in})
\begin{equation}\label{zeta}
I(\zeta,\lambda)=I(\omega,\lambda^A)\quad\text{for all $\lambda\in\mathcal E^+$},
\end{equation}
where $\lambda^A$ is the only measure in $\mathcal E'(A)$ meeting (\ref{n.e.}). But this follows at once by passing to the limits in (\ref{BalK}) when $K\uparrow A$ and utilizing \cite[Section~IV.1, Theorem~1]{B2}, (\ref{potconv}), and the relation (see \cite[Theorem~4.8]{Z-arx1})
\[U^{\lambda^K}\uparrow U^{\lambda^A}\quad\text{pointwise on $X$ as $K\uparrow A$}.\]

Exploiting Lemma~\ref{l2} once again, we derive from (\ref{balK}) and (\ref{oconv}) that, actually,\footnote{It has been used here that (\ref{balA}) only needs to be verified on compact subsets of $A$. Besides, the countable subadditivity of inner capacity on universally measurable sets has been utilized as well.}
\begin{equation}\label{balA}
 U^\zeta=U^\omega\quad\text{n.e.\ on $A$}.
\end{equation}
Assume (\ref{balA}) also holds for some $\theta\in\mathcal E'_H(A)$ in place of $\zeta$. Then, by Lemma~\ref{str},
\[U^\theta=U^\zeta\quad\text{n.e.\ on $A$},\]
and so, by virtue of (\ref{L}) and the characteristic property of $\zeta^A$ given by (\ref{n.e.}),
\[\theta=\zeta^A=\zeta.\]
Thus, $\zeta$ is the only measure in $\mathcal E'_H(A)$ whose potential coincides with $U^\omega$ n.e.\ on $A$. Since, according to (\ref{zeta}), $\zeta=\omega^A$, this verifies (i).

In view of (\ref{balA}), $\zeta\in\Gamma_{A,\omega}$, the class $\Gamma_{A,\omega}$ being introduced by means of (\ref{G}). As $\zeta=\omega^A$, (\ref{minpot}) will therefore follow once we show that for any given $\nu\in\Gamma_{A,\omega}$,
\begin{equation}\label{G1}
U^\nu\geqslant U^\zeta\quad\text{on $X$}.
\end{equation}

Combining (\ref{balA}) with $U^\nu\geqslant U^\omega$ n.e.\ on $A$ implies, by use of Lemma~\ref{str}, that
\[U^\nu\geqslant U^\zeta\quad\text{n.e.\ $A$},\]
and hence $\nu$ also belongs to the class $\Gamma_{A,\zeta}$, given by (\ref{G}) with $\zeta$ in place of $\omega$. Therefore, according to \cite[Theorem~3.1(a)]{Z-arx-22},
\begin{equation*}
U^\nu\geqslant U^{\zeta^A}\quad\text{on $X$},
\end{equation*}
whence (\ref{G1}), for $\zeta^A=\zeta$ by virtue of (\ref{L}) with $\lambda:=\zeta\in\mathcal E'_H(A)$.

To complete the proof of (iii), assume (\ref{minpot}) also holds for some $\theta\in\Gamma_{A,\omega}$ in place of $\zeta=\omega^A\in\Gamma_{A,\omega}$. Then, clearly, $U^\zeta\geqslant U^\theta\geqslant U^\zeta$ on $X$, whence $\theta=\zeta$ (Lemma~\ref{l0}).

The proof of (\ref{minen}) runs in a manner similar to that of (\ref{minpot}), the only difference being in applying \cite[Definition~3.1]{Z-arx-22} in place of \cite[Theorem~3.1(a)]{Z-arx-22}. Noting by utilizing Lemma~\ref{str} that the class $\Gamma_{A,\omega}$ is convex, we finally conclude that the solution to problem (\ref{minen}) is unique, which follows with the aid of standard arguments based on the parallelogram identity in the pre-Hil\-bert space $\mathcal E$ and the strict positive definiteness of the kernel $\kappa$. This establishes (iv), whence the theorem.

\subsection{Proof of Theorem~\ref{th2}}\label{sec-proof2}

We first prove two preparatory lemmas. Observe that in Lemma~\ref{laux2}, only the perfectness of
the kernel in question is, in fact, used.

\begin{lemma}\label{laux2} For any $\mu\in\mathcal E^+(A)$,
\begin{equation}\label{cK}
  \mu|_K\to\mu\quad\text{strongly in $\mathcal E^+(A)$ as $K\uparrow A$}.
  \end{equation}
\end{lemma}

\begin{proof} As $\mu|_A=\mu$ (Section~\ref{sec-intr1}), \cite[Lemma~1.2.2]{F1} with positive $g\in C_0(X)$ gives
\begin{equation}\label{vconv}
 \mu|_K\to\mu\quad\text{vaguely as $K\uparrow A$}.
\end{equation}
Further, since $\kappa\geqslant0$, the net $\bigl(U^{\mu|_K}\bigr)$ increases pointwise on $X$ as $K\uparrow A$, whence
\[\|\mu|_K\|^2\leqslant\bigl\langle\mu|_K,\mu|_{K'}\bigr\rangle\quad\text{for all $K'\geqslant K$},\]
and therefore
\begin{equation}\label{2}\|\mu|_K-\mu|_{K'}\|^2\leqslant\|\mu|_{K'}\|^2-\|\mu|_K\|^2.\end{equation}
But the net $\bigl(\|\mu|_K\|\bigr)$ also increases as $K\uparrow A$ and does not exceed $\|\mu\|$,\footnote{It holds true, actually, that $\|\mu|_K\|\uparrow\|\mu\|$ as $K\uparrow A$, which is clear from (\ref{vconv}) on account of the vague lower semicontinuity of $I(\nu)$ on $\nu\in\mathfrak M^+$ (the principle of descent \cite[Lemma~2.2.1(e)]{F1}).} hence it is Cauchy in $\mathbb R$.
In view of (\ref{2}), this implies that the net $(\mu|_K)_{K\in\mathfrak C_A}$ is strong Cauchy in $\mathcal E^+(A)$. But then,
according to the definition of a perfect kernel, this net must converge strongly to its vague limit, which together with (\ref{vconv}) establishes (\ref{cK}).
\end{proof}

In the rest of this section, the assumptions of Theorem~\ref{th2} are required to hold. Then, clearly, $U^\omega$ is bounded and continuous on $\overline{A}$, and moreover
\begin{equation}\label{liminfty}
\lim_{y\to\infty_X}\,U^\omega(y)=0.
\end{equation}

\begin{lemma}\label{laux1}
If a net $(\mu_s)_{s\in S}\subset\mathcal E^+(\overline{A})$ converges strongly to $\mu_0$, and
\begin{equation}\label{mum}
\mu_s(X)\leqslant M<\infty\quad\text{for all $s\in S$},
\end{equation}
then
\begin{equation*}
\lim_{s\in S}\,I_f(\mu_s)=I_f(\mu_0).
\end{equation*}
\end{lemma}

\begin{proof}
Since $\|\mu_s\|\to\|\mu_0\|$ as $s$ ranges over $S$, we only need to show that
\begin{equation}\label{pr2'}
I(\mu_0,\omega)=\lim_{s\in S}\,I(\mu_s,\omega).
\end{equation}
To this end, we first note that, because of the perfectness of the kernel, $\mu_s\to\mu_0$ also vaguely, and hence, due to (\ref{mum}), \begin{equation}\label{mum'}\mu_0(X)\leqslant M,\end{equation}
the mapping $\nu\mapsto\nu(X)$ being vaguely l.s.c.\ on $\mathfrak M^+$ \cite[Section~IV.1, Proposition~4]{B2}.

By virtue of (\ref{liminfty}), for any given $\varepsilon>0$, there is a compact set $K_0\subset X$ such that
\[U^\omega(x)<\varepsilon\quad\text{for all $x\not\in K_0$}.\]  On account of (\ref{mum}) and (\ref{mum'}), we therefore get\footnote{In (\ref{SSS'}) and (\ref{SSS''}), we have utilized \cite[Section~IV.4]{B2} (see Proposition~2 as well as Corollary~2 to Theorem~1).\label{foot}}
\begin{equation}\label{SSS'}
\biggl|\int U^\omega(x)1_{K_0^c}(x)\,d(\mu_s-\mu_0)(x)\biggr|<2M\varepsilon\quad\text{for all $s$}.
\end{equation}

The above $K_0$ can certainly be chosen so that
\[K_0\cap\overline{A}\ne\varnothing.\]
As $U^\omega$ is continuous on $\overline{A}$, while any compact subspace of $X$ is normal \cite[Section~IX.4, Proposition~1]{B3}, applying the Tietze-Urysohn extension
theorem \cite[Theorem~0.2.13]{E2} shows that there exists positive $\varphi\in C_0(X)$ such that
\[\varphi(x)=U^\omega(x)\quad\text{if $x\in K_0\cap\overline{A}$},\qquad
\varphi(x)\leqslant\varepsilon\quad\text{if $x\in\overline{A}\setminus K_0$},\]
which implies, in turn, that for all $s$ large enough,
\begin{align}&\biggl|\int U^\omega(x)1_{K_0}(x)\,d(\mu_s-\mu_0)(x)\biggr|=\biggl|\int \bigl(\varphi-\varphi|_{K^c_0}\bigr)\,d(\mu_s-\mu_0)\biggr|\notag\\
{}&\leqslant\biggl|\int\varphi\,d(\mu_s-\mu_0)\biggr|+\int \varphi|_{K^c_0}\,d(\mu_s+\mu_0)<3M\varepsilon.\label{SSS''}\end{align}
(Here we have used (\ref{mum}), (\ref{mum'}), and the fact that $\mu_s\to\mu_0$ vaguely; cf.\ also footnote~\ref{foot}.)
Combining (\ref{SSS'}) and (\ref{SSS''}) gives (\ref{pr2'}), whence the lemma.
\end{proof}

To prove (v), we first verify that the measure $\zeta$, uniquely determined by means of (\ref{oconv}), solves problem (\ref{hatw}). Since, in consequence of (\ref{oconv}) and Lemma~\ref{laux1},
\[I_f(\zeta)=\lim_{K\uparrow A}\,I_f(\omega^K),\]
this will follow once we show that for any given $\mu\in\mathcal E'_H(A)$,
\begin{equation*}
I_f(\mu)\geqslant\lim_{K\uparrow A}\,I_f(\omega^K).\end{equation*}

Choose a sequence $(\mu_j)\subset\mathcal E^+_H(A)$ converging strongly to this $\mu$, and $(K_j)\subset\mathfrak C_A$ with the property $\|\mu_j-\mu_j|_{K_j}\|<j^{-1}$ for all $j\in\mathbb N$;
such $K_j$ do exist by virtue of Lemma~\ref{laux2} applied to $\mu_j$. Then obviously $(\mu_j|_{K_j})\subset\mathcal E^+_H(A)$, and moreover $\mu_j|_{K_j}\to\mu$ strongly in $\mathcal E'_H(A)$. By making use of Lemma~\ref{laux1} we therefore get
\[I_f(\mu)=\lim_{j\to\infty}\,I_f(\mu_j|_{K_j})\geqslant\lim_{j\to\infty}\,I_f(\omega^{K_j})=\lim_{K\uparrow A}\,I_f(\omega^K),\]
the inequality being clear from (\ref{GaussK}) applied to $K_j$.

It has thus been shown that $\zeta$ solves, indeed, problem (\ref{hatw}). Such a solution is unique, which can be easily verified by means of standard arguments based on the convexity of the class $\mathcal E'_H(A)$, the energy principle, and a pre-Hil\-bert structure on the space $\mathcal E$. Since, according to (\ref{zeta}), $\zeta=\omega^A$, the proof of (v) is complete.

\section{Some further properties of balayage}\label{sec-further}

As before, a locally compact space $X$ and a set $A\subset X$ are arbitrary, whereas a kernel $\kappa$ and a measure $\omega\in\mathfrak M^+$ satisfy the permanent requirements (a)--(c).

\begin{proposition}\label{bal-pot} We have
\begin{align}\label{in-pot}
&U^{\omega^A}\leqslant U^\omega\quad\text{on all of $X$},\\
&I(\omega^A)=I(\omega^A,\omega)\leqslant I(\omega),\label{in-en}\\
&\omega^A(X)\leqslant h\omega(X),\label{emm'}\end{align}
where $h$ is the constant appearing in the $h$-Ugaheri maximum principle.
\end{proposition}

\begin{proof} For every $K\in\mathfrak C_A$, we have $U^{\omega^K}=U^\omega$ n.e.\ on $K$ (see (\ref{n.e.'}) for $A:=K$), hence $\omega^K$-a.e.\ (Lemma~\ref{l1}), and therefore, by the domination principle,
\begin{equation}\label{in-pK}
U^{\omega^K}\leqslant U^\omega\quad\text{on $X$}.
\end{equation}
Letting here $K\uparrow A$, we get (\ref{in-pot}) by making use of (\ref{c3}). Alternatively, (\ref{in-pK}) results in (\ref{in-pot}) in view of (\ref{c2}) and the principle of descent \cite[Lemma~2.2.1(b)]{F1}.

Integrating (\ref{in-pot}) with respect to $\omega$ gives $I(\omega^A,\omega)\leqslant I(\omega)$. Noting from the preceding paragraph that $I(\omega^K)=I(\omega^K,\omega)$ for all $K\in\mathfrak C_A$, we obtain the equality in (\ref{in-en}) with the aid of the limit relations $\|\omega^K\|\to\|\omega^A\|$ as $K\uparrow A$, cf.\ (\ref{c1}), and
\[\lim_{K\uparrow A}\,\int U^{\omega^K}\,d\omega=\int U^{\omega^A}\,d\omega,\]
the latter being derived from (\ref{c3}) by applying \cite[Section~IV.1, Theorem~1]{B2}.

The remaining inequality (\ref{emm'}) obviously holds, for $\omega^A\in\mathcal E'_H(A)$ (Theorem~\ref{th1}), $H$ being given by (\ref{H}), whereas $\mu(X)\leqslant H$ for all $\mu\in\mathcal E'_H(A)$, cf.\ (\ref{supm}).\end{proof}

\begin{proposition}\label{bal-rest}For arbitrary $Q\subset A$,\footnote{As in Landkof \cite[p.~264]{L}, (\ref{rest}) might be referred to as "{\it the inner balayage with a rest}".}
\begin{equation}\label{rest}
\omega^Q=(\omega^A)^Q,\end{equation}
and therefore
\begin{equation}\label{rest'}U^{\omega^Q}\leqslant U^{\omega^A}\quad\text{on $X$}.\end{equation}
\end{proposition}

\begin{proof} We can certainly assume that $\cp_*Q>0$, since otherwise $\omega^Q=(\omega^A)^Q=0$.

As (a)--(c) do hold for $Q$ along with $A$, Theorem~\ref{th1}(i) and Remark~\ref{REM} show that $\omega^Q$ does exist, and it is uniquely determined within $\mathcal E'(Q)$ by the equality
\begin{equation*}\label{rest1}U^{\omega^Q}=U^\omega\quad\text{n.e.\ on $Q$}.\end{equation*}
On the other hand, $\omega^A\in\mathcal E^+$ (Theorem~\ref{th1}); therefore, according to \cite[Theorem~4.3]{Z-arx1}, $(\omega^A)^Q$ does exist as well, and it is the only measure in $\mathcal E'(Q)$ having the property
\begin{equation}\label{rest2}U^{(\omega^A)^Q}=U^{\omega^A}=U^\omega\quad\text{n.e.\ on $Q$},\end{equation}
the latter equality in (\ref{rest2}) being derived from (\ref{n.e.'}) by making use of Lemma~\ref{str}. Comparing these two assertions implies (\ref{rest}).

The remaining inequality (\ref{rest'}) then follows at once from \cite[Eq.~(4.6)]{Z-arx1}, applied to $\mu:=\omega^A\in\mathcal E^+$. Alternatively, (\ref{rest'}) can be deduced from the equality
\[U^{\omega^A}(x)=\sup_{K\in\mathfrak C_A}\,U^{\omega^K}(x)\quad\text{for all $x\in X$},\]
cf.\ (\ref{c3}), by noting that $\mathfrak C_Q\subset\mathfrak C_A$.
\end{proof}

In Propositions~\ref{bal-tot-m}, \ref{th-tot} and Lemma~\ref{l-eq}, we assume that Frostman's maximum principle is fulfilled (i.e.\ $h=1$, cf.\ (a)).

\begin{proposition}\label{bal-tot-m} If moreover $X$ is $\sigma$-compact, then
$\omega^A$ is of minimum total mass in the class $\Gamma_{A,\omega}$; that is,
\begin{equation}\label{eq-t-m}\omega^A(X)=\min_{\mu\in\Gamma_{A,\omega}}\,\mu(X).\end{equation}
\end{proposition}

\begin{proof} Since $\omega^A\in\Gamma_{A,\omega}$ (Theorem~\ref{th1}), we are reduced to showing that
\begin{equation*}\omega^A(X)\leqslant\mu(X)\quad\text{for any $\mu\in\Gamma_{A,\omega}$.}\end{equation*}
But for such $\mu$, we have $U^{\omega^A}\leqslant U^\mu$ on all of $X$ (Theorem~\ref{th1}(iii)),
and the claimed inequality then follows at once by applying Theorem~\ref{pr-pos}.
\end{proof}

\begin{remark}\label{t-m-nonun} The extremal property (\ref{eq-t-m}) cannot, however, serve as an alternative characterization of the inner balayage $\omega^A$, for it does not determine $\omega^A$ uniquely. Indeed, consider the $\alpha$-Riesz kernel $|x-y|^{\alpha-n}$ of order $\alpha\leqslant2$, $\alpha<n$, on $\mathbb R^n$, $n\geqslant2$, a proper closed subset $A$ of $\mathbb R^n$ that is {\it not $\alpha$-thin at infinity},\footnote{By Kurokawa and Mizuta \cite{KM}, a set $Q\subset\mathbb R^n$ is said to be {\it inner $\alpha$-thin at infinity} if
\begin{equation*}
 \sum_{j\in\mathbb N}\,\frac{\cp_*(Q_j)}{q^{j(n-\alpha)}}<\infty,
 \end{equation*}
where $q\in(1,\infty)$ and $Q_j:=Q\cap\{x\in\mathbb R^n:\ q^j\leqslant|x|<q^{j+1}\}$. See also \cite[Section~2]{Z-bal2}.} and let $\omega\ne0$ be a positive measure of {\it finite} energy with ${\rm Supp}(\omega)\subset A^c$. Then
\begin{equation}\label{eq-t-m1}\omega^A\ne\omega\quad\text{and}\quad\omega^A(\mathbb R^n)=\omega(\mathbb R^n),\end{equation}
where the former is obvious because ${\rm Supp}(\omega^A)\subset A$ while the latter holds true by virtue of \cite[Corollary~5.3]{Z-bal2}.
Noting that $\omega,\omega^A\in\Gamma_{A,\omega}$, we conclude from (\ref{eq-t-m}) and (\ref{eq-t-m1}) that there are actually infinitely many measures of minimum total mass in the class $\Gamma_{A,\omega}$, for so is any one of the form $a\omega+b\omega^A$, where $a,b\in[0,1]$ and $a+b=1$.
\end{remark}

Before providing a formula for evaluation of the total mass $\omega^A(X)$ of the inner swept measure $\omega^A$ (Proposition~\ref{th-tot}), we first analyze the continuity of the inner equilibrium potential $U^{\gamma_A}$ under the exhaustion of $A$ by compact subsets $K$.

\begin{lemma}\label{l-eq}For any $A\subset X$ with $\cp_*A<\infty$,
\begin{equation}\label{later-eq}U^{\gamma_K}\uparrow U^{\gamma_A}\quad\text{pointwise on $X$ as $K\uparrow A$},\end{equation}
where $\gamma_K$, resp.\ $\gamma_A$, denotes the inner equilibrium measure of $K$, resp.\ of $A$.\footnote{See Fuglede \cite[Section~4.1]{F1} (in particular, Theorem~4.1 and the remarks attached to it). Here we have utilized the fact that, due to the energy principle,
$\cp K<\infty$ for any compact $K\subset X$.}\end{lemma}

\begin{proof}Since $\cp K\uparrow\cp_*A$ as $K\uparrow A$ \cite[p.~163]{F1}, we conclude in much the same way as in \cite[Proof of Theorem~4.1]{F1} that
\begin{equation}\label{conv-eq-m}\gamma_K\to\gamma_A\quad\text{strongly (hence, vaguely) in $\mathcal E^+$ as $K\uparrow A$}.\end{equation}
But for any $K,K'\in\mathfrak C_A$ such that $K\leqslant K'$,
\[1=U^{\gamma_K}=U^{\gamma_{K'}}=U^{\gamma_A}\quad\text{n.e.\ on $K$}\]
(cf.\ \cite[p.~175, Remarks]{F1}), hence $\gamma_K$-a.e.\ (Lemma~\ref{l1}). Therefore, by the domination principle, the net $(U^{\gamma_K})_{K\in\mathfrak C_A}$ increases pointwise on $X$ to some function that does not exceed  $U^{\gamma_A}$. To establish (\ref{later-eq}), we thus only need to verify the inequality
\[U^{\gamma_A}\leqslant\lim_{K\uparrow A}\,U^{\gamma_K}\quad\text{on all of $X$},\]
which is obvious from (\ref{conv-eq-m}) by the principle of descent \cite[Lemma~2.2.1(b)]{F1}.
\end{proof}

\begin{proposition}\label{th-tot}For any $A\subset X$ with $\cp_*A<\infty$,\footnote{For the $\alpha$-Riesz kernel of order $\alpha\in(0,2]$, $\alpha<n$, on $\mathbb R^n$, $n\geqslant2$, Proposition~\ref{th-tot} remains valid for any $\omega\in\mathfrak M^+$ and any $A\subset\mathbb R^n$ that is inner $\alpha$-thin at infinity (even if $\cp_*A=\infty$), which is seen by combining Theorems~2.2 and 5.1 from \cite{Z-bal2}.}
\begin{equation}\label{eq-mass1}\omega^A(X)=\int U^{\gamma_A}\,d\omega.\end{equation}
\end{proposition}

\begin{proof} Relations (\ref{rest}) and \cite[Eq.~(7.2)]{Z-arx1}, the latter being applied to $\mu:=\omega^A\in\mathcal E^+$, give
\[\omega^K(X)=(\omega^A)^K(X)\leqslant\omega^A(X)\quad\text{for every $K\in\mathfrak C_A$},\]
hence
\begin{equation}
\limsup_{K\uparrow A}\,\omega^K(X)\leqslant\omega^A(X)\leqslant\liminf_{K\uparrow A}\,\omega^K(X),\label{line}\end{equation}
where the latter inequality is obtained from (\ref{c2}) by use of the vague lower semicontinuity of the mapping $\nu\mapsto\nu(X)$ on $\mathfrak M^+$ \cite[Section~IV.1, Proposition~4]{B2}.

But for compact $K$, by Lebesgue--Fubini's theorem \cite[Section~V.8, Theorem~1]{B2},
\begin{equation*}\omega^K(X)=\int1\,d\omega^K=\int U^{\gamma_K}\,d\omega^K=\int U^{\omega^K}\,d\gamma_K=\int U^\omega\,d\gamma_K=\int U^{\gamma_K}\,d\omega,\end{equation*}
for $U^{\gamma_K}=1$ (resp.\ $U^{\omega^K}=U^\omega$) holds true n.e.\ on $K$, hence $\omega^K$-a.e.\ (resp.\ $\gamma_K$-a.e.).
By substituting this into (\ref{line}) we therefore obtain
\[\omega^A(X)=\lim_{K\uparrow A}\,\omega^K(X)=\lim_{K\uparrow A}\,\int U^{\gamma_K}\,d\omega.\]
Since the net $(U^{\gamma_K})_{K\in\mathfrak C_A}$ increases pointwise on $X$ to $U^{\gamma_A}$ (Lemma~\ref{l-eq}), we arrive at (\ref{eq-mass1}) by applying \cite[Section~IV.1, Theorem~1]{B2} to the integral on the right.
\end{proof}

\begin{proposition}\label{pr-outer} Assume that the space $X$ is $\sigma$-compact and perfectly normal, while the sets under consideration are Borel. Then Propositions~{\rm\ref{bal-pot}--\ref{bal-tot-m}} and {\rm\ref{th-tot}} remain valid with the inner concepts replaced throughout by the outer ones.
\end{proposition}

\begin{proof}
  This can be easily seen by applying Theorems~\ref{th1'}, \ref{l-top}, \cite[Theorem~9.4]{Z-arx1}, and \cite[Theorem~4.3]{F1}.
\end{proof}

\begin{theorem}\label{pr-in-ou} If $X$ is second-countable, then for arbitrary $A\subset X$, there exists a $K_\sigma$-set $A_0\subset A$ such that\footnote{Compare with \cite[Proposition~VI.2.2]{BH} by  Bliedtner and Hansen, showing that the outer balayage to any $A\subset X$ can always be reduced to the balayage to a $G_\delta$-set $A'\supset A$, the balayage in \cite{BH} being treating in the setting of balayage spaces. Recall that $Q\subset X$ is said to be a $K_\sigma$-{\it set}, resp.\ a $G_\delta$-{\it set}, if it is a countable union of compact subsets, resp.\ a countable intersection of open sets.} \begin{equation}\label{OO}\omega^A=\omega^{A_0}=\omega^{*A_0}.\end{equation}
\end{theorem}

\begin{proof}
As seen from (\ref{c3}), the net $(U^{\omega^K})_{K\in\mathfrak C_A}$ increases pointwise on $X$ to $U^{\omega^A}$. The space $X$ being Hausdorff and sec\-ond-count\-able while the functions $U^{\omega^K}$ being l.s.c.\ on $X$, applying \cite[Appendix~VIII, Theorem~2]{Doob} shows that there exists an increasing sequence $(K_j)_{j\in\mathbb N}$ of compact subsets of $A$ having the property
\begin{equation*}U^{\omega^{K_j}}\uparrow U^{\omega^A}\quad\text{pointwise on $X$ as $j\to\infty$}.\end{equation*}
Therefore, for the $K_\sigma$-set $A_0$,
\begin{equation*}A_0:=\bigcup_{j\in\mathbb N}\,K_j,\end{equation*}
we obtain, again by making use of (\ref{c3}),
\begin{equation*}U^{\omega^{A_0}}=\lim_{j\to\infty}\,U^{\omega^{K_j}}=U^{\omega^A}\quad\text{on all of }X.\end{equation*}
Both $\omega^{A_0}$ and $\omega^A$ being of finite energy, this gives $\omega^A=\omega^{A_0}$ (Lemma~\ref{l0}).
Besides, since the second-countable, locally compact space $X$ is $\sigma$-compact and perfectly normal (footnote~\ref{FOOT}),  while $A_0$ is Borel, we also have $\omega^{A_0}=\omega^{*A_0}$ (Theorem~\ref{th1'}). This completes the proof of (\ref{OO}), whence the theorem.
\end{proof}

\section{Balayage of signed measures}\label{sec-sign} Defining the inner and the outer balayage of $\xi\in\mathfrak M$ to arbitrary $A\subset X$ by
\begin{equation}\label{s1}
  \xi^A:=(\xi^+)^A-(\xi^-)^A,\quad\xi^{*A}:=(\xi^+)^{*A}-(\xi^-)^{*A},
\end{equation}
we assume henceforth that $\xi^+$ and $\xi^-$, the positive and negative parts of $\xi$ in the Hahn--Jor\-dan decomposition \cite[Section~III.1, Theorem~2]{B2}, meet (b) and (c).

\begin{theorem}\label{ths} Then, the inner balayage $\xi^A$ does exist, and it is uniquely characterized within $\mathcal E$ by the symmetry relation\footnote{Given $\theta\in\mathcal E$, $\theta^A$ does exist, for $\theta\in\mathcal E\iff\theta^\pm\in\mathcal E^+$ (see \cite[Section~3.1]{F1}). With regard to $\lambda^A$ and $\lambda^{*A}$, where $\lambda\in\mathcal E^+$, see \cite{Z-arx1} (Theorems~4.3 and 9.4).}
\begin{equation}\label{sym-s}I(\xi^A,\theta)=I(\theta^A,\xi)\quad\text{for all $\theta\in\mathcal E$}.\end{equation}
If moreover $X$ is $\sigma$-compact and perfectly normal, while $A$ is Borel, then all this remains valid with $\xi^A$ and $\theta^A$ replaced by $\xi^{*A}$ and $\theta^{*A}$, respectively.
\end{theorem}

\begin{proof} In view of (\ref{s1}), $\xi^A$ does indeed exist, since so do both $(\xi^+)^A$ and $(\xi^-)^A$ (see Theorem~\ref{th1} with $\omega:=\xi^\pm$). Moreover, (\ref{sym-s}) holds true, which is clear from (\ref{sym}) with $\omega:=\xi^\pm$ and $\lambda:=\theta^\pm$ by the bilinearity of mutual energies. To verify the statement on the uniqueness, suppose that for some $\zeta\in\mathcal E$, $I(\zeta,\theta)=I(\theta^A,\xi)$ for all $\theta\in\mathcal E$. Subtracting this from (\ref{sym-s}) gives $I(\xi^A-\zeta,\theta)=0$ for all $\theta\in\mathcal E$, which implies, by substituting $\theta:=\xi^A-\zeta$, that $I(\xi^A-\zeta)=0$, whence $\xi^A=\zeta$ by the energy principle.

Assume now that $X$ is $\sigma$-compact and perfectly normal, while $A$ is Borel. Then the latter part of the theorem follows immediately from the former on account of the equalities $\xi^{*A}=\xi^A$ and $\theta^{*A}=\theta^A$ (cf.\ (\ref{s1}), (\ref{BB}), and (\ref{eq-bal-oUU})).
\end{proof}

In the rest of this paper, along with (a) and the above-mentioned assumptions on $\xi\in\mathfrak M$, the following (g) and (h) are required to hold:
\begin{itemize}
\item[(g)] {\it $X$ is second-countable.}
\item[(h)] {\it The set of all $\varphi\in C_0(X)$ representable as potentials $U^\theta$ of measures $\theta\in\mathcal E$ is dense in the space $C_0(X)$ equipped with the inductive limit topology.}
\end{itemize}
These requirements on $X$, $\kappa$, $A$, and $\xi$ will usually not be repeated henceforth.

Then, as seen from Theorem~\ref{l-alt-countt} below, the characteristic property (\ref{sym-s}) needs only to be verified for certain {\it countably many} $\theta\in\mathcal E$, depending on $X$ and $\kappa$ only.

\begin{theorem}\label{l-alt-countt}There exists a countable set $\mathcal E^\circ\subset\mathcal E$, depending on $X$ and $\kappa$ only,  such that the inner balayage $\xi^A$ is uniquely determined within $\mathcal E$ by the relation
\begin{equation}\label{rel1}I(\xi^A,\theta)=I(\theta^A,\xi)\quad\text{for all $\theta\in\mathcal E^\circ$}.\end{equation}
That is, if for some $\zeta\in\mathcal E$,
\begin{equation}\label{rel2}I(\zeta,\theta)=I(\theta^A,\xi)\quad\text{for all $\theta\in\mathcal E^\circ$,}\end{equation}
then
\[\zeta=\xi^A.\]
If $A=B$ is Borel, then the same remains valid with $\xi^A$ and $\theta^A$ replaced by the outer swept measures $\xi^{*B}$ and $\theta^{*B}$, respectively.
\end{theorem}

\begin{remark}
  We shall now provide three examples, where (a) and (h) do hold, whereas (g) is obvious.

    In the first, $\kappa$ is the $\alpha$-Riesz kernel $|x-y|^{\alpha-n}$ of order $\alpha\leqslant2$, $\alpha<n$, on $X:=\mathbb R^n$, $n\geqslant2$. Note that for every $\varphi\in C_0(\mathbb R^n)$, there exist a compact set $K\subset\mathbb R^n$ and a sequence $(\varphi_j)\subset C_0^\infty(\mathbb R^n)$ (obtained by regularization \cite[p.~22]{S}) such that all the $\varphi$ and $\varphi_j$ equal $0$ on $K^c$, and moreover $\varphi_j\to\varphi$ uniformly on $K$ (hence, also in the inductive limit topology on $C_0(\mathbb R^n)$, cf.\ Lemma~\ref{foot-conv} below). Since each $\varphi\in C_0^\infty(\mathbb R^n)$ can be represented as the $\alpha$-Riesz potential of a (signed) measure on $\mathbb R^n$ of finite $\alpha$-Riesz energy (see \cite[Lemma~1.1]{L} and \cite[Lemma~3.2]{Z-bal2}), (h) indeed holds.

    In the next two examples, $X:=D$, where $D\subset\mathbb R^n$, $n\geqslant2$, is open. Then (h) holds if either $\kappa$ is the ($2$-)Green kernel for the Laplace operator on Greenian $D$, or $\kappa$ is the $\alpha$-Green kernel of order $\alpha\in(1,2)$ for the fractional Laplacian on bounded $D$ of class $C^{1,1}$.\footnote{$D$ is said to be of class $C^{1,1}$ if for every $y\in\partial_{\mathbb R^n}D$, there exist $B(x,r)\subset D$ and $B(x',r)\subset D^c$, where $r>0$, that are tangent at $y$, see \cite[p.~458]{Bogdan}. Here $B(x,r):=\{z\in\mathbb R^n:\ |z-x|<r\}$.}
 This follows by applying \cite[p.~75, Remark]{L}, resp.\ \cite[Eq.~(19)]{Bogdan}, to $\varphi\in C_0^\infty(D)$, and then utilizing the same approximation technique as just above.

    Regarding the validity of (a) in each of these three examples, see Remark~\ref{(a)}.
\end{remark}

\section{Proof of Theorem~\ref{l-alt-countt}}

\subsection{Preliminaries}\label{sec-prel} According to Bourbaki \cite[Section~III.1.1]{B2}, {\it the inductive limit topology} on $C_0(X)$ is the inductive limit $\mathcal T$
of the locally convex topologies of the spaces $C_0(K;X)$, where $K$ ranges over all compact subsets of $X$, while $C_0(K;X)$ is the space of all $\varphi\in C_0(X)$ with ${\rm Supp}(\varphi)\subset K$, equipped with the topology $\mathcal T_K$ of uniform convergence on $K$. Thus, by \cite[Section~II.4, Proposition~5]{B4}, $\mathcal T$ is the finest of the locally convex topologies on $C_0(X)$ for which all the canonical injections $C_0(K;X)\to C_0(X)$, $K\subset X$ being compact, are continuous.

\begin{lemma}[{\rm see \cite[Section~III.1, Proposition~1(i)]{B2}}]\label{aux1}For any compact $K\subset X$, the topology on the space $C_0(K;X)$ induced by $\mathcal T$ is identical with the topology $\mathcal T_K$.\end{lemma}

Due to the assumption (g), $X$ is $\sigma$-compact \cite[Section~IX.2, Corollary to Proposition~16]{B3}, and hence there is a sequence of relatively compact open subsets $U_j$ with the union $X$ and such that $\overline{U}_j\subset U_{j+1}$, see \cite[Section~I.9, Proposition~15]{B1}. The space $C_0(X)$ is then the {\it strict} inductive limit of the sequence of spaces $C_0(\overline{U}_j;X)$, cf.\ \cite[Section~II.4.6]{B4}, for the topology induced on $C_0(\overline{U}_j;X)$ by $\mathcal T_{\overline{U}_{j+1}}$ is just $\mathcal T_{\overline{U}_j}$. Hence, by \cite[Section~II.4, Proposition~9]{B4}, $C_0(X)$ is Hausdorff and complete (in $\mathcal T$).

\begin{lemma}\label{foot-conv}For any sequence $(\varphi_k)\subset C_0(X)$, {\rm(i$_1$)} and {\rm(ii$_1$)} are equivalent.
\begin{itemize}\item[{\rm(i$_1$)}] $(\varphi_k)$ converges to $0$ in the strict inductive limit topology $\mathcal T$.
\item[{\rm(ii$_1$)}]There exists a compact subset $K$ of $X$ such that ${\rm Supp}(\varphi_k)\subset K$ for all $k$, and $(\varphi_k)$ converges to $0$ uniformly on $K$.\end{itemize}
\end{lemma}

\begin{proof} Assume $(\varphi_k)\subset C_0(X)$ approaches $0$ in $\mathcal T$.
Since $\{\varphi_k: k\in\mathbb N\}$ is then bounded in $\mathcal T$, there exists a compact set $K\subset X$ such that ${\rm Supp}(\varphi_k)\subset K$ for all $k$ (see \cite[Section~III.1, Proposition~2(ii)]{B2}).\footnote{The cited proposition from \cite{B2} is applicable here, for a locally compact, $\sigma$-compact space is paracompact \cite[Section~I.9, Theorem~5]{B1}.} Applying now Lemma~\ref{aux1} we therefore conclude that $(\varphi_k)$ also approaches $0$ in the (Hausdorff) topology $\mathcal T_K$, and so (i$_1$)$\Longrightarrow$(ii$_1$).

Since $\mathcal T$ is Hausdorff as well, the opposite follows directly from Lemma~\ref{aux1}.\end{proof}

In Lemma~\ref{foot-conv}, only the assumption (g) was, in fact, used; whereas Theorem~\ref{th-dense}, crucial to the proof of Theorem~\ref{l-alt-countt}, is based substantially on both (g) and (h).

\begin{theorem}\label{th-dense}
 There is a countable set $C_0^\circ\subset C_0(X)$, depending on $X$ and $\kappa$ only, which is dense in $C_0(X)$ in the  topology $\mathcal T$; therefore, for any two $\mu,\nu\in\mathfrak M$,
 \begin{equation}\label{D}
 \mu=\nu\iff\mu(\varphi)=\nu(\varphi)\quad\text{for every $\varphi\in C_0^\circ$}.
 \end{equation}
 \end{theorem}

\begin{proof} We first note that, due to (g), there is a countable set $L\subset C_0(X)$, depending on $X$ only and having the following
property: for any given $\varphi\in C_0(X)$, there exist a sequence $(\psi_j)\subset L$ and a positive function $\psi_0\in L$ such that, for every number
$\varepsilon>0$,
\[|\varphi-\psi_j|<\varepsilon\psi_0\quad\text{for all $j\geqslant j_0$}\]
(see \cite[Section~V.3.1, Lemma]{B2}). This implies that for those $\varphi\in C_0(X)$ and $(\psi_j)\subset L$, there is a compact set $K\subset X$ such that all the $\varphi$ and $\psi_j$ equal $0$ on $K^c$, while $\psi_j\to\varphi$ uniformly on $K$; therefore, $\psi_j\to\varphi$ also in $\mathcal T$ (Lemma~\ref{foot-conv}).
Thus, the countable set $L\subset C_0(X)$ is dense in $C_0(X)$, equipped with the strict inductive topology $\mathcal T$.

But, by (h), the set $C_0(X)\cap\{U^\theta:\ \theta\in\mathcal E\}$ is also dense in $C_0(X)$ (in $\mathcal T$). Hence, for every $\psi\in L$, $L$ being introduced just above, there exists a sequence $(\theta^\psi_p)_{p\in\mathbb N}\subset\mathcal E$ such that $\bigl(U^{\theta^\psi_p}\bigr)_{p\in\mathbb N}\subset C_0(X)$ while
\[U^{\theta^\psi_p}\to\psi\quad\text{in $\mathcal T$ as $p\to\infty$.}\]
All this indicates that
\begin{equation}\mathcal E^\circ:=\bigl\{\theta^\psi_p: \ \psi\in L, \ p\in\mathbb N\bigr\}\label{e}\end{equation}
is a countable subset of $\mathcal E$, depending on $X$ and $\kappa$ only, and moreover
\begin{equation}\label{C0}C_0^\circ:=\{U^\theta:\ \theta\in\mathcal E^\circ\}\end{equation}
is a dense subset of the space $C_0(X)$, equipped with the topology $\mathcal T$.

The remaining assertion (\ref{D}) follows by applying \cite[Section~III.1.7]{B2}.
\end{proof}

\subsection{Proof of Theorem~\ref{l-alt-countt}}\label{pr-last} Let $X$, $\kappa$, $\xi$, and $A$ be as indicated in Section~\ref{sec-sign}, and let $\mathcal E^\circ\subset\mathcal E$ be given by (\ref{e}). Then $\mathcal E^\circ$ depends on $X$ and $\kappa$ only (see the proof of Theorem~\ref{th-dense}), and moreover (\ref{rel1}) holds in view of (\ref{sym-s}) with $\theta\in\mathcal E^\circ$. To show that the inner balayage $\xi^A$ is uniquely characterized within $\mathcal E$ by means of (\ref{rel1}), assume that for some $\zeta\in\mathcal E$, (\ref{rel2}) takes place. Subtracting (\ref{rel2}) from (\ref{rel1}) gives \[I(\zeta,\theta)=I(\xi^A,\theta)\quad\text{for all $\theta\in\mathcal E^\circ$},\]
or equivalently
\[\zeta(\varphi)=\xi^A(\varphi)\quad\text{for all $\varphi\in C_0^\circ$},\]
$C_0^\circ$ being introduced by means of (\ref{C0}), and consequently $\zeta=\xi^A$ (Theorem~\ref{th-dense}).

To verify the remaining claim, note that, due to (g), $X$ is $\sigma$-compact and perfectly normal (footnote~\ref{FOOT}). Hence, if $A$ is Borel, then, by virtue of (\ref{s1}),  (\ref{BB}), and (\ref{eq-bal-oUU}),
\[\xi^{*A}=\xi^A\quad\text{and}\quad\theta^{*A}=\theta^A,\]  which substituted into the former part of the theorem finalizes the whole proof.

\section{Acknowledgements} This research was supported in part by a grant from the Simons Foundation (1030291, N.V.Z.). I am also deeply grateful to Krzysztof Bogdan for helpful discussions around \cite{Bogdan}, and to the authors of \cite{Dr0} for drawing my attention to \cite{KM}.

\end{document}